\documentclass[a4]{amsart}
\usepackage{hyperref}
\usepackage{cleveref}
\usepackage{geometry}
\usepackage{xcolor}
\usepackage{mathrsfs}
\usepackage{amsmath,amssymb,dsfont}
\usepackage{enumitem}

\allowdisplaybreaks

\theoremstyle{cupthm}
\newtheorem{thm}{Theorem}[section]

\newtheorem{cor}[thm]{Corollary}
\newtheorem{lemma}[thm]{Lemma}
\theoremstyle{cupdefn}
\newtheorem{defn}[thm]{Definition}

\theoremstyle{cuprem}
\newtheorem{rem}[thm]{Remark}

\numberwithin{equation}{section}
\newcommand{\Hsquare}{%
	\text{\fboxsep=-.1pt\fbox{\rule{0pt}{1ex}\rule{1ex}{0pt}}}%
}

\newcommand{\obp}{\otimes^{\gamma}}

\newcommand{\sA}{\mathscr A}

\newcommand{\mA}{\mathcal{A}}
\newcommand{\mB}{\mathcal B}

\theoremstyle{definition}
\newtheorem{theorem}{Theorem}

\newtheorem{definition}{Definition}
\newcommand{\vv}{\vert\vert}

\begin{document}

	\title{Geometry of Banach algebra $\mA$ and the Bidual of $L^1(G)\hat{\otimes}\mA$}
		\author{Lav Kumar Singh} \keywords{ Banach
			algebras, Arens regularity, projective
			tensor product, Topological Center, Group algebra.}
		
		\subjclass[2010]{47B10, 46B28, 46M05}
		
		\thanks{This research was supported by the Intitute post-doctoral fellowship provided to the author by the Indian Institute of Science Education and Research-Bhopal(India) }
		\email{lavksingh@hotmail.com}
	\maketitle

	\begin{abstract}
		This article is intended towards the study of the bidual of generalized group algebra $L^1(G,\mA)$ equipped with two Arens product, where $G$ is any locally compact group and $\mA$ is a Banach algebra. We show that the left topological center of $(L^1(G)\hat\otimes\mA)^{**}$ is a Banach $L^1(G)$-module if $G$ is abelian. Further it also holds permanance property with respect to the unitization of $\mA$. We then use this fact to extend the remarkable result of A.M Lau and V. Losert\cite{Lau-losert}, about the topological center of $L^1(G)^{**}$ being just $L^1(G)$, to the reflexive Banach algebra valued case using the theory of vector measures.  We further explore pseudo-center of $L^1(G,\mA)$ for non-reflexive Banach algebras $\mA$ and give a partial characterization for elements of pseudo-center using the Cohen's factorization theorem. In the running we also observe few consequences when $\mA$ holds the Radon-Nikodym property and weak sequential completeness.
	\end{abstract}
\section{Introduction}
Arens regularity is an important tool for classifying Banach algebras. R. Arens in \cite{Arens} defined two products $\Hsquare$ and $\diamond$ on the bidual of a Banach algebra $\mA$. When these two products are same, the Banach algebra $\mA$ is said to be Arens regular. He went on to prove that the algebra $\ell^1(\mathbb Z)$ equipped with point-wise multiplication is Arens regular. Further it can easily be seen that every reflexive Banach algebra is Arens regular. Remarkably, every $C^*$-algebra is Arens regular (\cite{Palacios}). On the other end, for any locally compact infinite group $G$, the Banach algebra $L^1(G)$ is known to be Arens irregular. {\"U}lger in \cite{Ulger} showed that the projective tensor product of Arens regular Banach algebras may not be regular. but if $\mA\hat\otimes\mB$ is Arens regular then both $\mA$ and $\mB$ must be Arens regular. He also explored the connections between geometric properties (like RNP, WCG, WSC etc) of the underlying Banach space of an algebra $\mA$ and its Arens regularity in his subsequent articles. A.M Lau and V. Losert showed in \cite{Lau-losert} that the left topological center of $L^1(G)^{**}$ is exactly $L^1(G)$ for any locally compact group $G$. \\

We consider the vector valued Banach algebra  $L^1(G,\mA)=L^1(G)\hat\otimes \mA$, where $\mA$ is a Banach algebra and $G$ is a locally compact group. The natural question- whether the topological center of $L^1(G,\mA)^{**}$ is $L^1(G,\mA)$ remains un-answered even for the case when $\mA$ is finite dimensional Banach algebra, due to the lack of structure theorem for finite dimensional Banach algberas. We explore this aspect in this article and settle the proof that for a compact abelian group $G$ and reflexive Banach algebra $\mA$, the Banach algebra $L^1(G,\mA)$ is Strongly Arens irregular. To do so, we first show that the topological center  $Z(L^1(G,\mA)^{**})$ is an $L^1(G)$-module and holds permanance property with respect to unitization. This will enable us to assume that $\mA$ is a unital algebra and avoid approximate idenitity arguments at each step. Many of the techniques in the proof are motivated from Lau and Losert's methods in \cite{Lau-losert}, with significant generalization to the vector valued setting. In the last section we give a nice application of Cohen's factorization theorem to give a partial characterization of elements in pseudo-center(defined in section 2). Finally we end the article with a  weak sequential completeness(WSC) like property for Banach space $\mA$ which have RNP and are WSC. For studying the vector valued functions spaces, it is inevitable to encounter Bochner integrals and vector measures and hence we revisit  the basics in the next section.
\section{Preliminaries to Vector valued functions and integration}
Let $(S,\mathscr{A},\mu)$ be a measure space and $X$ be a Banach space. A function $h:S\to X$ is said to be {\em $\mu$-simple} if $h=\sum_{i=1}^n\chi_{A_i}x_i$, where $A_i\in \mathscr{A}$ for each $i$ such that $\mu(A_i)<\infty$ and $x_1,x_2,..,x_n\in X$. Now if a function $f:S\to X$ is said to be {\em $\mu$-strongly measurable} if there exists a sequence of $\mu$-simple functions converging pointwise to $f$ almost everywhere. The Pettis measurability theorem states that $f$ is $\mu$-strongly measurable if and only if it is $\mu$-Borel measurable and $\mu$-essentially separably valued (see \cite{Ryan} and \cite{Tuomas} for more details).
For each $\mu$-simple function $h=\sum_{i=1}^r\chi_{A_i}x_i$, we define $\int hd\mu=\sum_{i=1}^r\mu(A_i)x_i$.
A $\mu$-strongly measurable function $f:S\to X$ is said to be {\em Bochner integrable} with respect to $\mu$ if there exists a sequence of $\mu$-simple functions $f_n:S\to X$ such that $
\lim_n \int||f(s)-f_n(s)||d\mu \to 0$. And the Bochner integral is given by $$\int fd\mu=\lim_n \int f_nd\mu$$
For $1\leq p<\infty$ we denote by $L^p(S,X)$ the collection of equivalance classes of $\mu$-strongly measurable functions $f$ which are $\mu$-almost everywhere equal and $\int||f(s)||^pd\mu<\infty$. $L^p(S,X)$ becomes a Banach space with respect to the norm $||f||_{L^p(S,X)}=\left(\int ||f(s)||^pd\mu\right)^{1/p}$. The space $L^\infty(S,X)$ is the collection of equivalence classes of functions which are essentially bounded in the natural sense. $L^\infty(S,X)$ is also a Banach space with respect to the norm $$||f||_{L^\infty(S,X)}=\inf\{M~:~\mu\left(\{s:~||f(s)||\leq M~\text{almost everywhere}\}\right)=0\}$$
\begin{definition}
	A function $F:\mathscr{A}\to X$ is called an {\em $X$-valued measure} on $S,\mathscr{A}$ if it is countably additive in the sense that for any disjoint sequence $\{A_i\}$ of sets in $\mathscr A$, we have $F(\cup A_i)=\sum F(A_i)$, where convergence on the right hand side is given in norm. The {\em variation} of an $X$-valued measure  is the map $||F||:S\to [0,\infty]$ given by $$||F||(A)=sup\Big\{\sum_{i=1}^n||F(A_i)||~:~A_i\in \mathscr A, \cup_{i=1}^n A_i=A~\text{and~} A_i\cap A_j=\phi \text{~for~}i\neq j, n\in \mathbb N\Big\}$$ We say that $F$ has bounded variation if $||F||(S)<\infty$. Further, $F$ is said to be absolutely continuous with respect to $\mu$ if for any $A\in \mathscr{A}$ with $\mu(A)=0$, we have $F(A)=0$. We denote the Banach space of all $X$-valued measures on $(S,\mathscr A )$ of bounded variation by $M(S,X)$ with respect to the variation norm $||F||=||F||(S)$. Further, $B(S,X)$ denote the closed subspace of $M(S,X)$ consisting of $X$-valued measures of bounded variation which are aboslutely coninuous with respect to $\mu$. A $X$-valued measure $F$ is said to be {\em regular} if $\varphi\mu$ is a regular measure for each $\phi\in X^*$. Equivalently for every set $E\in \mathscr{A}$ and $\epsilon>0$, there exists closed set $A\in\mathscr A$ contained in $E$ and an open set $B\in\mathscr A$ containing $E$ such that $||F(E\setminus A)||<\epsilon$ and $||F(B\setminus E)||<\epsilon$. We denote the Banach subspace of all regular measures in $M(S,X)$ by $M_r(S,X)$. Clearly $B(S,X)$ is a subspace of $M_r(S,X)$ if $X$ has RNP.
	\begin{theorem} \cite[Section 5.3]{Ryan}\label{RRT1}
If $(S,\mathscr B(S),\mu)$ is a compact Hausdorff measure space and $X$ is a Banach space then, $C(S,X)^*$ can be isometrically identified with $M_r(S,X^*)$ of all regular measures of bounded variation on Borel susbsets of $S$.
	\end{theorem}
	\begin{proof}
		An element $\phi$ of $C(S,X)^*$ gives rise to a family of $X^*$-valued measures on $S$ in the following way.  Fixing $\xi \in X$, one can define a linear functional $L_{\phi,\xi}$ on $C(S)$ by sending the function $f$ on $S$, to the value of $\phi$ on the function $S \to X$ given by $s \mapsto f(x) \xi$. 
		$$
		L_{\phi,\xi}(f) = \phi\left(f\otimes \xi\right).
		$$
		From the usual Riesz theorem, there is then a measure $\mu_{\phi, \xi}$ defined on the Borel subsets of $S$ satisfying
		$$
		L_{\xi,\phi}(f) = \int_S f \, d\mu_{\phi, \xi}.
		$$
	Corresponding to each pair $\phi\in C(S,X)$ and $\xi\in X$, we have obtained a measure on $S$..  
		
		Now define a map $\mu_{\phi}$ from the Borel subsets of $S$ to $X^*$ as follows: for any Borel subset $E$ of $S$, define $m_{\phi}(E)$ to be the linear functional on $X$ given by
		$$
		\mu_{\phi}(E)(\xi) = \int_E  \, d\mu_{\phi, \xi}.
		$$
		The map $\mu_{\phi}$ is regular $X^*$-valued measure on $S$.  Since the functions of the form $x \mapsto f(x) \xi$, with $f \in C(S)$ and $\xi \in X$, are dense in $C(S,X)$, it is easy to show that $\phi$ is uniquely determined by $m_{\phi}$. 
		
	Conversely, starting with a $X^*$-valued regular measure, one can show that it must be $m_{\phi}$ for some $\phi$ in $C(X,Y)^*$. Thus, $\phi\mapsto m_\phi$ is an isometric isomorphism with respect to the variation norm.
	\end{proof}
	Notice that for any $\phi\in L^1(S,X)$, the function $F:\mathscr{A}\to X$ defined as $$F(A)=\int_A\phi(s)d\mu(s)$$ for each $A\in \mathscr{A}$ is a well defined vector measure of bounded variation  which is absolutely continuous with respect to $\mu$ and $||F||(A)=\int_A||\phi(s)||d\mu(s)$. It turns out that if $X$ has a nice geometric property called Radon-Nikodym Property(RNP), then all vector measures of bounded variation that is absolutely continuous with respect to $\mu$ arises in the same fashion.
\end{definition}
\begin{definition}
	A Banach space $X$ is said to have the Radon-Nikodym Property(RNP) with respect to a measure space $(S,\sA,\mu)$, if for every $X$-valued measure of $F$ of bounded variation on $(S,\sA)$ that is absolutely continuous with respect to $\mu$, there exists a function $\phi\in L^1(S,X)$ such that $$F(A)=\int_A\phi d\mu,~~~~~ A\in \sA$$
\end{definition}
Thus, if $X$ has RNP then $B(S,X)$ can be isometrically identified with $L^1(G,X)$.
\begin{definition}
	A bounded operator $T:L^1(S)\to X$ is said to representable if there exists a function $\phi\in L^\infty(X,S)$ such that $$Tf=\int_S f\phi d\mu,~~~~f\in L^1(S,X)$$
\end{definition}
\begin{theorem} \cite[Th. 1.3.10]{Tuomas}\label{RNP1}
	Let $(S,\sA,\mu)$ be a $\sigma$-finite measure space and $X$ be a Banach space, and let $1\leq p<\infty$ and $\frac{1}{p}+\frac{1}{q}=1$. The following assertions are equivalenet. 
	\begin{enumerate}
		\item $X^*$ has RNP with respect to $(S,\sA,\mu)$.
		\item The mapping $g\mapsto\phi_g$ establishes an isometric isomorphism of Banach space $$ L_q(S,X^*)\simeq(L^p(S,X))^* $$
		where $\left<\phi_g,f\right>=\int\left<g(s),f(s)\right>d\mu(s)$ for each $f\in L^p(S,X)$.

	\end{enumerate}
\begin{rem}
	If $X^*$ does not have the RNP w.r.t $(S,\mathscr{A},\mu)$, then $g\to \phi_g$ is still an isometry onto a norming subspace of $L^p(S,X)^*$,  but it may not be surjective in general.
\end{rem}
\end{theorem}
\begin{theorem}\cite[Th. 1.3.15]{Tuomas}\label{representable}
	Let $(S,\sA,\mu)$ be a $\sigma$-finite measure space. For a Banach space $X$, the following assertions are equivalent.
	\begin{enumerate}
		\item $X$ has RNP with respect to $(S,\sA,\mu)$.
		\item Every bounded linear operator $T:L^1(S)\to X$ is representable.
	\end{enumerate}
It is a standard fact that if $X$ is reflexive or is a separable dual space, then $X$ posess RNP with respect to any $\sigma$-finite measure space. The fourier algebra $A(G)$ has RNP iff $G$ is compact. The algebra of Trace class operators on a Hilbert space Space has the RNP. Dual of a $C^*$-algebra $\mA$ have RNP iff and only if $\mA$  is Scattered. Pre-dual of a Von-neumann algebra has RNP if and only if it is a direct sum of type-$I$ factors (see \cite{Chu}). Space $c_0,\ell^\infty, C[0,1]$ and  $L^1([0,1])$ does not have RNP. Proof of these facts can be found in any standard text on geometry of Banach spaces.

The RNP defined above is dependent is relative to a measure space. Interestingly RNP is more well behaved with respect to $\sigma$-finite measure spaces.
\begin{theorem}\cite[Th. 1.3.26]{Tuomas}
	For a Banach space $X$ the following are equivalent.
	\begin{enumerate}
		\item $X$ has RNP with respect to $[0,1]$.
		\item $X$ has RNP with respect to any $\sigma$-finite measure space.
	\end{enumerate}
In view of the above theorem, we will say that $X$ has RNP without refererring to a measure space whenever the measure space involved is $\sigma$-finite.
\end{theorem}

\begin{definition}
	A Banach space $X$ is said to be {\em Weakly Sequentially Complete}(WSC), if every weakly Cauchy sequence $\{\phi_n\}_{n=1}^\infty$ of functions in $X$ converges weakly to some element in $X$.
\end{definition}
It is a well known and easy to prove fact that $L^1(S)$ is WSC for any measure space $S$. All reflexive spaces are WSC by default. If $X$ is WSC and $\mu$ is a finite measure, then $L^1(S,X)$ is also WSC \cite[Th. 11]{Talagrand}. The fourier algebra $A(G)$ is WSC for any locally compact Hausdorff group. A $C^*$ algebra is WSC if and only if is finite dimensional.
\end{theorem}
Now, if $\mA$ is a Banach algebra and $G$ is any locally compact Hausdorff group, then $L^1(G,\mA)$ can be given an algebra structure through convolution $f\ast g(t)=\int_G f(s)g(s^{-1}t)d\mu(s)$, with respect to which $L^1(G,\mA)$ becomes a Banach algebra. This generalized group algebra can be identified with the projective tensor product $L^1(G)\hat{\otimes} \mA$ (see \cite{Kaniuth} for the proof).

We shall make use of the Cohen's factorization theorem as stated below, in coming sections.
\begin{theorem}[Cohen's Factorization]\label{Cohen}
	Let $\mathscr{A}$ be a Banach algebra and $\mathscr{K}$ be a left Banach $\mathscr{A}$-module and suppose that $\mathscr{A}$ has a left approximate identity $\{e_\beta\}_{\beta\in B}$ that is also a left approximate identity for $\mathscr{K}$ bounded by some constant $\delta\geq 1$. Then for any $x_o\in \mathscr{K}$ and any $\epsilon>0$, there exists an $a\in \mathscr A$ and $x\in \mathscr{K}$ such that $x_o=ax$, $||a||\leq \delta$ and $||x-x_o||\leq \epsilon$.
	
\end{theorem}
\section{Topological Center of $L^1(G,\mA)^{**}$ as $L^1(G)$-module. }
Given any Banach algebra $A$, due to Hahn-Banach, we have an isometric embedding $J:A\to A^{**}$ as Banach spaces. For each $a\in A$ and $f\in A^*$, we define two functionals $f_a,{_a}f\in A^*$ as $$f_a(b)=f(ab)~,~{_a}f(b)=f(ba)~~~\text{for all } b\in A$$
Further, for each $f\in A^*$ and  $m\in A^{**}$ we define $f_m, {_m}f\in A^*$ as $$f_m(a)=m({_a}f)~,~{_m}f(a)=m(f_a)~\forall a\in A$$
Now, the two Arens product on $A^**$  are defined as $$m\Hsquare n(f)=m({_n}f)~,~m\diamond n(f)=n(f_m)$$ for each $m,n\in A^{**}$ and $f\in A^*$.
It is a trivial fact that $A^{**}$ becomes a Banach algebra with respect both these Arens product $\Hsquare$ and $\diamond$ and the two Arens prodcuct agree on $J(A)$, and the embedding $J$ becomes an algebra Homomorphism with respect to both these products. We denote by $Z(A^{**})$ the (left) topological center- \begin{align}
	Z(A^{**})&=\{m\in A^{**}~:~m\Hsquare n=m\diamond n~\forall n\in A^{**}\}\\&=\{m\in L^1(G,\mA)^{**}~:~n\mapsto m\Hsquare n~\text{is~} w^\ast\text~\text{continuous}\}\nonumber
\end{align}
Clearly, $Z(A^{**})$ is a Banach algebra itself naturally.
For Banach algebras $\mA$ and $\mathcal B$, their direct sum $\mA\oplus \mathcal B$ is a Banach algebra when equipped with co-ordinate wise multiplication $(a,b)\cdot (c,d)=(ac,bd)$ and norm $||(a,b)||=||a||+||b||$. Further we have identification for dual Banach space $(A\oplus B)^*=A^*\oplus_\infty B^*$ and $$((\mA\oplus \mathcal B)^{**},\Hsquare)=(\mA^{**},\Hsquare)\oplus (\mathcal B^{**},\Hsquare)$$ $$((\mA\oplus \mathcal B)^{**},\diamond)=(\mA^{**},\diamond)\oplus (\mathcal B^{**},\diamond)$$
The following permanence property is elementary in nature but will help us significantly in calculating the topological center.
\begin{theorem}
Let $\mA$ and $\mathcal B$ be Banach algebras. Then $A\oplus B$ is SAI if and only if both $\mA$ and $\mathcal B$ are SAI.
\end{theorem}
Recall that the minimal unitization of a non-unital Banach algebra $\mA$ is $\tilde{\mA}=\mA\oplus \mathbb C$ equipped with norm $||(a,\alpha)||=||a||+|\alpha|$ and multiplication $(a,\alpha)(b,\beta)=(ab+\beta a+\alpha b,\alpha\beta)$.
	
If Banach algebra $\mA$ and $\mB$ are such that $\mA$ is a right Banach $\mB$-module, i.e $\exists$ continuous bilinear map $\mathfrak m:\mA\times \mB\to \mA$ of norm one giving an algebraic module structure, then we form a direct sum Banach algebra $\mA\tilde{\oplus}\mB$ with norm $||(a,b)||=||a||+||b||$ and multiplication $(a,b)(c,d)=(ac+bc+ad,bd)$. We can consider $L^1(G,\mA)$ as right $L^1(G)$-module naturally by $\mathfrak{m}:L^1(G,\mA)\times L^1(G)\to L^1(G,\mA)$ defined as $\mathfrak{m}(\phi,F)=\phi\ast F$.
\begin{theorem}\label{natural}
 The map	$\theta:L^1(G,\tilde{\mA})\to L^1(G,\mA)\tilde\oplus L^1(G)$ is an isometric isomorphism of algebras, where $\theta$ is defined as $\theta(\phi)=(\pi_1\phi,\pi_2\phi)$ for each $\phi\in L^1(G,\mA)$.
\end{theorem}
\begin{proof}
	Notice that
		\begin{align*} ||\theta(\phi)||&=||\pi_1\phi||+||\pi_2\phi||\\&=\int||\pi_2\phi(t)||dt+\int ||\pi_2\phi(t)||dt\\&=\int||\phi(t)||dt\\&=||\phi||_{L^1(G,\tilde{\mA})}
		\end{align*}
i.e $\theta$ is isometry. $\theta$ is surjective can be seen easily. Further, for any two $\phi_1,\phi_2\in L^1(G,\tilde{\mA})$, notice that
\begin{align*}
	\phi_1\ast\phi_2(t)&=\int\phi_1(s)\phi_2(s^{-1}t)dt\\&=\int\big(\pi_1\phi_1(s),\pi_2\phi_1(s)\big)\big(\pi_1\phi_2(s^{-1}t),\pi_2\phi_2(s^{-1}t)\big)dt\\&=\int \big(\pi_1\phi_1(s)\pi_1\phi_2(s^{-1}t)+\pi_1\phi_1(s)\pi_2\phi_2(s^{-1}t)+\pi_2\phi_1(s)\pi_1\phi_2(s^{-1}t),\pi_2\phi_1(s)\pi_2\phi_2(s^{-1}t)\big)dt\\&=\big(\pi_1\phi_1\ast\pi_1\phi_2(t)+\pi_1\phi_1\ast\pi_2\phi_2(t)+\pi_2\phi_1\ast\pi_1\phi_2(t),\pi_2\phi_1\ast\pi_2\phi_2(t)\big)
\end{align*}
Thus, we can see that $\theta(\phi_1\ast\phi_2)=\theta(\phi_1)\theta(\phi_2)$. Hence, $\theta$ is an algebra isomorphism.
\end{proof}
\begin{theorem}
	If $\mA$ is a Banach algebra and $G$ is a locally compact abelian group then	$Z(L^1(G,\mA)^{**})$ is an $L^1(G)$-module.
\end{theorem}
\begin{proof}Since $G$ is abelian, the algebra $L^1(G)$ is commutative.
	Consider the following adjoints of biliniar map $\mathfrak{m}$, where $\mathfrak{m}^*(f,\phi)(F)=\left<f,\mathfrak{m}(\phi,F)\right>$ 
	\begin{align*}
		\mathfrak{m}&:L^1(G,\mA)\times L^1(G)\to L^1(G,\mA)\\
		\mathfrak{m}^*&:L^1(G,\mA)^*\times L^1(G,\mA)\to L^1(G)^*\\
		\mathfrak{m}^{**}&:L^1(G)^{**}\times L^1(G,\mA)^*\to L^1(G,\mA)^*\\
		\mathfrak{m}^{***}&:L^1(G,\mA)^{**}\times L^1(G)^{**}\to L^1(G,\mA)^{**}
	\end{align*}
	Let $m\in Z( L^1(G,\mA)^{**})$, $n\in L^1(G,\mA)^{**}$, $F\in L^1(G)$ and $f\in L^1(G,\mA)^*$. Further let $\{m_\alpha\}_{\alpha\in \wedge_1}$ and $\{n_\beta\}_{\beta\in \wedge_2}$ be nets in $L^1(G,\mA)$ converging to $m$ and $n$ respectively in the $w^\ast$-topology of $L^1(G,\mA)^{**}$.Then
	\begin{align}
		\mathfrak{m}^{***}(m,F)\Hsquare n(f)&=\left<m,(\mathfrak{m}^{**}(F,{_n}f))\right>\nonumber\\&=\lim_\alpha \left<\mathfrak{m}^{**}(F,{_n}f),m_\alpha\right>\nonumber\\&=\lim_\alpha \left<\mathfrak{m}^{*}({_n}f,m_\alpha),F\right>\nonumber\\&=\lim_{\alpha}\left<{_n}f,m_\alpha\ast F\right>\nonumber\\&=\left<{_F}({_n}f),m\right>
	\end{align}
	Now, 
	\begin{align}
		\mathfrak{m}^{***}(m,F)\diamond n(f)&=\lim_\beta \mathfrak{m}^{***}(m,F)({_{n_\beta}}f)\nonumber\\&=\lim_\beta m(\mathfrak{m}^{**}(F,{_{n_\beta}}f))\nonumber\\&=\lim_\beta\lim_\alpha\left<\mathfrak{m}^{**}(F,{_{n_\beta}}f),m_\alpha\right>\nonumber\\&=\lim_\beta\lim_\alpha \left<F,\mathfrak{m}^*({_{n_\beta}f},m_\alpha)\right>\nonumber\\&=\lim_\beta\lim_\alpha\left<{_{n_\beta}}f,m_\alpha\ast F\right>\nonumber\\&=\lim_\beta\left<{_F}({_{n_\beta}}f),m\right>\nonumber\\&=\lim_\beta \left<{_{n_\beta}}({_F}f),m\right>&&(\because n_\beta\ast F=F\ast n_\beta)\nonumber\\&=\lim_\beta m\Hsquare n_\beta({_F}f)&&(\text{here we use ~}m\in Z(L^1(G,\mA)^{**}))\nonumber\\&=\left<{_n}({_F}f),m\right>\nonumber\\&=\left<{_F}({_n}f),m\right>&&(\because {_F}({_n}f)={_n}({_F}f) )
	\end{align}
	Thus we see that $\mathfrak{m}^{***}(m,F)\Hsquare n(f)=\mathfrak{m}^{***}(m,F)\diamond n(f)$ for all $f\in L^1(G,\mA)^*$. Hence, $\mathfrak{m}^{***}(m,F)\in Z(L^1(G,\mA)^{**})$. Further it is easy to verify that $\mathfrak{m}^{***}(m,F_1\ast F_2)=\mathfrak{m}^{***}(\mathfrak{m}^{***}(m,F_1),F_2)$ holds for each $m\in Z(L^1(G,\mA)^{**})$ and $F_1,F_2\in L^1(G)$. Thus, the restriction of $m^{***}$ to $Z(L^1(G,\mA)^{**})\times L^1(G)$ gives the required $L^1(G)$-module structure on $Z(L^1(G,\mA)^{**})$.
\end{proof}
\begin{theorem}
	If $G$ is a locally compact abelian group and $\mA$ is a Banach algebra, then $$Z(L^1(G,\tilde{\mA})^{**})\cong Z(L^1(G,\mA)^{**})\tilde{\oplus} L^1(G)$$
	is an isometric isomorphism of algebras.
\end{theorem}
\begin{proof}
Consider the double adjoint map of $\theta$ $$\theta^{**}:L^1(G,\tilde\mA)^{**}\to L^1(G,\mA)^{**}\oplus L^1(G)^{**}$$
Clearly $\theta^{**}$ is an isometric isomorphism of Banach spaces, because $\theta$ is. We claim that the restriction of $\theta^{**}$ to $Z(L^1(G,\tilde\mA)^{**})$ gives us the isometric isomorphism of algebras  $Z(L^1(G,\tilde\mA)^{**})$ and $Z(L^1(G,\mA)^{**})\tilde{\oplus} L^1(G)$. To see this, let $\tilde m\in Z(L^1(G,\mA)^{**})$, $n\in L^1(G,\mA)^{**}$ and $f\in L^1(G,\mA)^*$. It is easy to verify that
\begin{align}\label{lav}
	{_{(n,0)}}(f,0)(\phi,\psi)&={_n}f(\phi)+{_n}f(\psi)&&\forall (\phi,\psi)\in L^1(G,\mA)\tilde\oplus L^1(G)
\end{align}
Thus,
 \begin{align}\label{lav1}
	\tilde m\Hsquare(n,0)(f,0)&=\tilde m \left({_{(n,0)}}(f,0)\right)\nonumber\\&=\lim_\alpha {_{(n,0)}}(f,0)(\pi_1\theta (\tilde m_\alpha),\pi_2\theta(\tilde m_\alpha))\nonumber\\&=\pi_1\theta^{**}(\tilde m)\Hsquare n(f)+{_n}f(\pi_2\theta^{**}(\tilde m))&&(\text{using~}\cref{lav})
\end{align}
Similarly, one can prove that \begin{align}\label{lav2}
	\tilde m\diamond(n,0)(f,0)=\pi_1\theta^{**}(\tilde m)\diamond n(f)+{_n}f(\pi_2\theta^{**}(\tilde m))
\end{align}
Since, $\tilde{m}\in Z(L^1(G,\tilde\mA)^{**})$, from \cref{lav1} and \cref{lav2} we deduce that $\pi_1\theta^{**}\Hsquare n=\pi_1\theta^{**}\diamond n$ for each $n\in L^1(G,\mA)^{**}$ i.e $\pi_1\theta^{**}(\tilde m)\in Z(L^1(G,\mA)^{**})$.
Thus, $\pi_1\theta^{**}(\tilde m)\in Z(L^1(G,\mA)^{**})$. And similarly, one can show that $\pi_2\theta^{**}(\tilde m)\in Z(L^1(G)^{**})=L^1(G)$. Hence, $\theta^{**}_{\restriction Z(L^1(G,\mA)^{**})}$ is a well defined isometric linear map into $Z(L^1(G,\mA)^{**})\oplus L^1(G)$. Further, it can be easily verified that this map is surjective and is a Homomorphism with respect to $\tilde\oplus$ structure on the right side.
\end{proof}
\begin{cor}\label{unitization1}For a locally compact abelian group $G$ and a  Banach algebra $\mA$, the group algebra $L^1(G,\mA)$ is SAI if and only if $L^1(G,\tilde{\mA})$ is SAI.
\end{cor}
\begin{defn}For $L^1(G,\mA)$, the (left) {\em topological pseudo-center} is defined as \begin{eqnarray*}Z_s(L^1(G,\mA)^{**})&=&\{m\in L^1(G,\mA)^{**}~:~m\Hsquare n(f)=m\diamond n(f)~\forall f\in L^\infty(G,\mA^*), n\in L^1(G,\mA)^{**}\}\\&=&\{m\in L^1(G,\mA)^{**}~:~n\mapsto m\Hsquare n~\text{is~}\sigma(L^1(G,\mA)^{**},L^\infty(G,\mA^*)) \text{-continuous}\}
\end{eqnarray*}
\end{defn}
Clearly, $Z_s(L^1(G,\mA)^{**})=Z(L^1(G,\mA)^{**})$ if $\mA^*$ has RNP. In general the pseudo center is a bigger class. \\
 We can consider $L^1(G,\mA)$ as a subalgebra of both $L^1(G,(\mA^{**},\Hsquare))$ and $L^1(G,(\mA^{**},\diamond))$ naturally since $\mA$ is a subalgebra of both $(\mA^{**},\Hsquare)$ and $(\mA^{**},\diamond)$.
We notice the following consequence of $\mA^*$ having RNP..
\begin{theorem} Let  $\mA$ be a Banach algebra such that $\mA^*$ has RNP, then
	$\theta_\Hsquare: L^1(G,(\mA^{**},\Hsquare))\to \left(L^1(G,\mA)^{**},\Hsquare\right)$ and $\theta_\diamond: L^1(G,(\mA^{**},\diamond))\to \left(L^1(G,\mA)^{**},\diamond\right)$ are isometric homomorphisms, where $\theta_\Hsquare(\phi)(f)=\theta_\diamond(\phi)(f)=\int_G \left<f(t),\phi(t)\right>d\mu(t)$ for all $\phi\in L^1(G,\mA^{**})$ and $f\in L^\infty(G,\mA^*)$.
\end{theorem}
\begin{proof}
	Since, $\mA^*$ has RNP, we have the dual identification $L^1(G,\mA)^*=L^\infty(G,\mA^*)$.
	Clearly, $\theta_\Hsquare$ and $\theta_\diamond$ are well defined linear maps which are same at the Banach space level. Further, $||\theta_\Hsquare(\phi)||\leq ||\phi||_{L^1(G,\mA^{**})}$. Since, $\phi$ is Bochner integrable, there exists a sequence of $\mu$-simple functions $\{\psi_n\}$ in $L^1(G,\mA^{**})$ which converges pointwise to $\phi$ almost everywhere such that $\lim_n\int_G||\psi_n(t)-\phi(t)||d\mu(t)\to 0$. For $\epsilon>0$, fix a natural number $N$ such that $\int_G||\psi_Nt)-\phi(t)||d\mu(t)<\epsilon/3$. Let $\psi_N=\sum_{i=1}^r\chi_{A_i}a^{**}_i$ for some disjoint $A_1,A_2,..,A_r\in \mathscr{B}(G)$ with $\mu(A_i)<\infty$ for each $i$ and $a^{**}_1,...,a^{**}_r\in \mA^{**}$. For each $i$, choose a $b^*_i\in \mA^*$ such that $\left| a^{**}_i(b^*_i)-||a^{**}_i||\right|<\epsilon/3 $. Now define $f:G\to \mA^*$ as $f=\sum_{i=1}^r\chi_{A_i}b^*_i$. Clearly $f$ is a $\mu$-simple function and $f\in L^\infty(G,\mA^*)$ such that $||f||_{L^\infty(G,\mA^*)}\leq 1$. Now notice that \begin{eqnarray*}
		\left|\int_G\left<f(t),\phi(t)\right>d\mu(t)-\int_G||\phi(t)||d\mu(t)\right|&=&\left|\int_G\left<\sum_{i=1}^r\chi_{A_i}b^*_i,\phi(t)\right>-||\phi(t)||d\mu(t)\right|\\
		&\leq&\sum_{i=1}^r\int_{A_i}\bigl\lvert\left<b^*_i,\phi(t)\right>-||\phi(t)||\bigr\rvert d\mu(t)\\
		&\leq&\sum_{i=1}^r\int_{A_i}\bigl\lvert \left<b^*_i,\phi(t)\right>-||a^{**}_i||\bigr\rvert d\mu(t)+\sum_{i=1}^r\int_{A_i}\bigl\lvert||a^{**}_i||-||\phi(t)||\bigr\rvert d\mu(t)\\&\leq&\sum_{i=1}^r\int_{A_i}\left|\left<b_i^*,\phi(t)\right>-a^{**}_i(b^*_i)\right|d\mu(t)+\sum_{i=1}^r\int_{A_i}\bigl\lvert a^{**}_i(b^*_i)-||a^{**}_i||\bigr\rvert d\mu(t)+\frac{\epsilon}{3}\\&< &\sum_{i=1}^r\int_{A_i}||\phi(t)-a^{**}_i||d\mu(t)+\frac{\epsilon}{3}+\frac{\epsilon}{3}\\&=&\int_G||\phi(t)-\psi_n(t)||d\mu(t)+\frac{2\epsilon}{3}\\&
		<&\epsilon
	\end{eqnarray*}
	Thus, we see that $||\theta_\Hsquare(\phi)||=||\phi||_{L^1(G,\mA)}$ and $\theta_\Hsquare,\theta_\diamond$ are isometry.  
	Now we prove  that $\theta_\Hsquare$ and $\theta_\diamond$ are algebra homomorphisms. For $\phi_1,\phi_2\in L^1(G,(\mA^{**},\Hsquare))$ and $f\in L^\infty(G,\mA^*)$, we have
	\begin{eqnarray*}
		\theta_\Hsquare(\phi_1\ast\phi_2)(f)&=&\int\left<\phi_1\ast\phi_2(t),f(t)\right>d\mu(t)\\
		&=&\int\int\left<\phi_1(s)\Hsquare\phi_(s^{-1}t),f(t)\right>d\mu(s)d\mu(t)\\&=&\int\left<\phi_1(s),\int {_{\theta_\Hsquare(\phi_2)(s^{-1}t)}}(f(t))d\mu(t)\right>d\mu(s)\\&=&\int\left<\phi_1(s),{_{\theta_\Hsquare(\phi_2)}}f\right>d\mu(s)\\&=&\theta_\Hsquare(\phi_1)({_{\theta_\Hsquare(\phi_2)}}f)\\&=&\theta_\Hsquare(\phi_1)\Hsquare\theta_\Hsquare(\phi_2)(f)
	\end{eqnarray*}
	Thus, $\theta_\Hsquare(\phi_1\ast\phi_2)=\theta_\Hsquare(\phi_1)\Hsquare\theta_\Hsquare(\phi_2)$. Hence, $\theta_\Hsquare$ is an  homomorphism of algebras. Similarly, $\theta_\diamond$ is a homomorphism.
\end{proof}
\section{The bidual of $L^1(G,\mA)$ for compact group $G$.}\label{4}
In this section, hereafter we will assume $G$ to a be a compact Hausdorff group with normalized Haar measure $\mu$ and $\mA$ will denote a Banach algebra, unless stated otherwise.\\
\begin{lemma}\label{easy1}
If $f\in L^\infty(G,\mA^*)$ and $\phi\in L^1(G,\mA)$, then $f_\phi(t)=\int L_{\phi(s)}(f(st))ds$ for all $t\in G$.
\end{lemma}
\begin{proof}
	Suppose $\psi \in L^1(G,\mA)$. Then \begin{eqnarray*}
		f_\phi(\psi)&=&f(\phi\ast\psi)\\
		&=&\int \left<f(t),\phi\ast\psi(t)\right>dt\\&=&\int\left<f(t),\int\phi(s)\psi(s^{-1}t)ds\right>dt\\&=&\int\int\left<f(st),\phi(s)\psi(t)\right>dtds\\&=&\int\left<\int (f(st))_{\phi(s)}ds,\psi(t)\right>dt
	\end{eqnarray*}
Since, $\psi$ was arbitrary, we conclude that $f_\phi(t)=\int (f(st))_{\phi(s)}ds$
\end{proof}
\begin{lemma}\label{easy2}
If $f\in L^\infty(G,\mA^*)$ and $\phi\in L^1(G,\mA)$, we have ${_\phi}f(t)=\int {_{\phi(t^{-1}s)}}(f(s))ds$
\end{lemma}
\begin{proof}
	Proved similarly as previous lemma.
\end{proof}

We say that $f\in L^\infty(G,X)$ is left uniformly continuous if $||L_a(f)-f||\to 0$ as $a\to e$, where $L_a(f)(x)=f(ax)$. Similarly, $f$ is said to be right uniformly continuous if  $||R_a(f)-f||\to 0$ as $a\to e$. Let $LUC(G,X),RUC(G,X)$ denote the collection of all left/right uniformly continuous functions in $L^\infty(G,X)$. Clearly $LUC(G,X), RUC(G,X)$ are closed subspaces.

\begin{lemma}
	Continuous functions are left/right uniformly continuous, i.e $C(G,X)\subseteq LUC(G,X)\cap RUC(G,X)$ for any Banach space $X$ and a compact group $G$.
\end{lemma}
\begin{proof}
Proof runs ditto as in the case of scalar valued functions. See for instance \cite[Prop. 2.6]{Folland}.
\end{proof}
\begin{lemma}
	Let $G$ be a compact group and $X$ be a Banach space. If $f\in C(G,X)$ then $||L_z(f)-f||_{L^1(G,X)}$ and $||R_z(f)-f||_{L^1(G,X)}$ tends to $0$ as $z\to e$.
\end{lemma}
\begin{proof}
	Fix a compact neighborhood $V$ of $e$. Let $K=(supp(f))V^{-1}\cap V(supp(f))$. Then $K$ is compact and $L_z(f)$ is supported in $K$ when $z\in V$. Hence, $||L_z(f)-f||_{L^1(G,X)}\leq \mu(G)||L_zf-f||_{L^{\infty}(G,X)}\to 0$ as $z\to e$
\end{proof}
\begin{lemma}
Let $\mA$ be any Banach algebra.	If $f\in L^\infty(G,\mA^*)$ and $\phi\in C(G,A)$, then $f_\phi\in LUC(G,\mA^*)$.
\end{lemma}

\begin{proof}
For any $z\in G$, it is straighforward to verify that $L_z(f_\phi)=f_{\widetilde{L_z\tilde{\phi}}}$. For any $t\in G$, \begin{eqnarray*}
	\vv L_z(f_\phi)(t)-f_\phi(t)\vv&=&\left|\left|\int(f(st))_{\widetilde{L_z\tilde{\phi}}-\phi(s)}d\mu(s)\right|\right|\\
	&\leq&||f||_{L^\infty(G,\mA^*)}\int||(\widetilde{L_z\tilde{\phi}})(s)-\phi(s)||d\mu(s)\\&=&||f||_{L^\infty(G,\mA^*)}\int||\phi(sz^{-1})-\phi(s)||d\mu(s)
\end{eqnarray*}
Thus, $||L_z(f_\phi)-f_\phi||_{L^\infty(G,\mA^*)}\leq ||f||_{L^\infty(G,\mA^*)}\left|\left|R_{z^{-1}}\phi-\phi\right|\right|_{L^1(G,\mA)}$ and by previous lemma, $f_\phi$ is left uniformly continuous.
\end{proof}

For $f\in L^\infty(G,X)$ and $t\in G$, write $(tf)(x)=f(xt)$.
\begin{lemma}\label{imp1}
Let $m\in Z_s$ and $f\in L^\infty(G,\mA^*)$, then $f_m\in LUC(G,\mA^*)$ and $f_m(t)(a)=\left<m,f^{(t,a)}\right>$ for each $t\in G$ and $a\in A$, where $f^{(t,a)}(s)={_a}(f(st))$
\end{lemma}
\begin{proof}
Using the Goldstein's theorem and the denseness of compactly supported functions in $L^1(G,\mA)$, we choose a net $\{\phi_{\alpha}\}_{\alpha\in \wedge}$ of compactly supported functions converging to $m$ in the $weak^\ast$-topology of $L^1(G,A)^{**}$. Then, \begin{eqnarray*}
	\left<n,f_m\right>&=&\left<m\diamond n,f\right>=\left< m\Box n,f\right>=\left<m,{_n}f\right>\\&=&\lim_\alpha\left<\phi_\alpha,{_n}f\right>=\lim_\alpha\left<n,f_{\phi_\alpha}\right>
\end{eqnarray*}
for each $n\in L^1(G,\mA)^{**}$. Hence, $\{f_{\phi_\alpha}\}$ converges weakly to $f_m$. Using Mazur's lemma, we can choose a net of  suitable convex combinations of $\{\phi_\alpha\}_{\alpha\in \wedge}$ such that $f_{\phi_\alpha}$ converges to $f_m$ in norm. Thus we assume that $f_{\phi_\alpha}$ converges to $f_m$ in norm. But $f_{\phi_\alpha}\in LUC(G,\mA^*)$ by previous lemma. Thus, it follows that $f_m\in LUC(G,\mA^*)$. Further, for any $t\in G$ and $a\in \mA$
\begin{eqnarray*}
	f_m(t)(a)&=&\lim_\alpha\int L_{\phi_\alpha(s)}(f(st))(a)d\mu(s)\\&=&\lim_{\alpha}\int f(st)(\phi_\alpha(s)a)d\mu(s)\\&=&\lim_{\alpha}\int\left<f(st),\phi_\alpha(s)a\right>d\mu(s)\\&=&\lim_{\alpha}\left<\phi_\alpha,f^{(t,a)}\right>\\&=&\left<m,f^{(t,a)}\right>
\end{eqnarray*} 
\end{proof}

\begin{lemma} \label{Imp2}
	Let $\mA$ be a unital Banach algebra. If $m\in Z_s$ is such that $m(f)=0$ for all $f\in C(G,\mA^*)$ then $m(f)=0$ for all $f\in L^\infty(G,\mA^*)$. Further if $\mA^*$ has RNP then $m=0$.
\end{lemma}
\begin{proof}
	 To see this, let $f\in L^\infty(G,\mA^*)$. Using the \cref{imp1}, for $\epsilon>0$, we choose $V\subset \{x:||f_m(x)-f_m(e)||<\epsilon\}$ such that $V$ is open and relatively compact. Consider $v=\frac{1_V}{\mu(V)}\otimes \mathds{1}_A$. One can easily see that ${_v}f\in C(G,\mA^*)$. Hence, ${_v}(f_m)(\phi)=({_v}f)_m(\phi)=m({_{\phi\ast v}}f)=0$, because ${_{\phi\ast v}}f\in C(G,\mA^*)$ for each $\phi\in L^1(G,\mA)$. Thus, ${_v}(f_m)=0$.
	\begin{eqnarray*}
		|m(f)|&=&\left|f_m(e)(\mathds{1}_\mA)-{_v}(f_m)(e)(\mathds{1}_\mA)\right|\\&=&\frac{1}{\mu(V)}\left|\int_V\left(f_m(e)(\mathds{1}_\mA)-f_m(x)(\mathds{1}_\mA)\right)dx\right|\\&\leq& \frac{1}{\mu(V)}||f_m(e)-f_m(x)||\mu(V)\\&=&\epsilon
	\end{eqnarray*}
Hence, $m(f)=0$ for all $f\in L^\infty(G,\mA^*)$, proving the first part of assertion. Now if $\mA^*$ has RNP then $L^\infty(G,\mA^*)$ is the full dual space of $L^1(G,\mA)$ and hence $m=0$.
\end{proof}

 Let $S^\infty(G,\mA^*)$ denotes the closure of the space of all $\mu$-simple functions in $L^\infty(G,\mA^*)$. Clearly $C(G,\mA^*)$ is contained in $S^\infty(G,\mA^*)$ . For each $\nu\in M_r(G,A)$ we define $\varphi_{\nu}:S(G,\mA^*)\to \mathbb C$ such that  for $\sum_{i=1}^r\chi_{E_i}\otimes a^*_i\in  S(G,\mA^*)$.  $$\left<\varphi_{\nu},\sum_{i=1}^r\chi_{E_i}\otimes a^*_i\right>=\sum_{i=1}^r\left<\nu(E_i),a^*_i\right>$$ It is an easy exercise to verify that this action is well defined(independent of representation of simple functions) and is linear. Further, it can be verified that  $||\varphi_{\nu}||=||\nu||$. Hence, $M_r(G,A)$ sits inside $S^\infty(G,\mA^*)^*$ isometrically.

\begin{lemma}
Let $\mA$ be a reflexive Banach algebra. Suppose that $m\in Z(L^1(G,\mA)^{**})$ and $\nu\in M_r(G,\mA)$ be such that there exists a sequence $\{\nu_n\}\in L^1(G,\mA)$ converging to $\nu$ in the $\sigma(M_r(G,\mA),C(G,\mA^*))$ topology. Then $m\diamond n_\nu\in L^1(G,\mA)$ for any continuous extension $n_\nu$ of $\nu$ to $L^1(G,\mA)^*$ .
\end{lemma}
\begin{proof}
Let us first assume that $u\in L^1(G,\mA)$. Then by \cref{RRT1}, restriction of $m$ to $C(G,\mA^*)$ is given by a measure $\eta\in M_r(G,\mA)$. For $f\in C_0(G,\mA^*)$
\begin{align*}
\left<m\Hsquare u,f\right> &=\left<m,{_u}f\right>\\&=\left<\eta,{_u}f\right>
\end{align*}
Since, regular $\mA$-valued measures are weakly compact in the sense that the associated operator $C(K)\to X$ is weakly compact (see \cite[Th. 5.2]{Ryan}), and space of weakly compact $\mA$-valued measures form a Banach algebra with respect to the convolution (see \cite[Th. 3.2]{White}), we conclude that $\left<m\Hsquare u,f\right> =\left<\eta\ast u,f\right>$ for all $f\in C(G,\mA^*)$. Notice that $m\Hsquare u\in Z(L^1(G,\mA)^{**})$. Thus, by \cref{Imp2},  $m\Hsquare u=\eta\ast u\in L^1(G,\mA)$ (because $\eta\ast u$ is a regular measure of bounded variation and $\nu\ast u\ll \mu$). \\
Now if $f\in L^\infty(G,\mA^*)$, then $f_m\in LUC(G)$ by \cref{imp1} and hence
\begin{align*}
	\left<m\diamond n_\nu,f\right>&=n(f_m)\\&=\lim_n \nu_n(f_m)\\&=\lim_n m\diamond \nu_n(f)\\&=\lim_n m\Hsquare \nu_n(f)
\end{align*}
But $m\Hsquare \nu_n\in L^1(G,\mA)$  as proved above and $L^1(G,\mA)$ is WSC, hence $m\diamond n_\nu\in L^1(G,\mA)$.
\end{proof}
\begin{cor}\label{metrizable}
Let $\mA$ be a unital refelexive Banach algebra $K$ be a closed subgroup of $G$ such that $G/K$ is metrizable, then $m\diamond \mu_k\in L^1(G,\mA)$ where $\mu_k$ is the $\mA$-valued vector measure on Borel subsets of $G$ such that $\mu_k(E)=\frac{\mu(E\cap K)}{\mu(K)}\mathds{1}_\mA$ for each Borel subset $E$ of $G$.
\end{cor}
\begin{proof}
Since, $G/K$ is metrizable, we can choose a decreasing sequence $\{U_n\}$ of neighborhoods of $K$ such that $K=\cap_n U_n$. Let $u_n=\frac{1}{\mu(U_n)}\chi_{U_n}\mathds 1_\mA$. Clearly $u_n\in L^1(G,\mA)$ and $u_n\to \mu_K$ in the $\sigma(M_r(G,\mA),C(G,\mA^*))$ topology. By previous lemma, $m\diamond \mu_k\in L^1(G,\mA)$.
\end{proof}
For a subgroup $K$ of $G$, we say that $f\in L^\infty(G,\mA^*)$ is right {\em $K$-periodic } if $kf=f$ for all $k\in K$.
\begin{lemma}\label{periodic}
	Let $K$ be a compact subgroup of $G$ and $m\in Z$. If $f\in L^\infty(G,\mA^*)$ is $K$-periodic then $\left<m,f\right>=\left<m\diamond\mu_K,f\right>$ for all $f\in L^\infty(G,\mA^*)$.
\end{lemma}
\begin{proof}
	Since, $f$ is $K$-periodic, by \cref{imp1}, $f_m$ is also $K$-periodic. Hence, \begin{align*}
		\left<m,f\right>&=f_m(e)(\mathds 1_\mA)\\&=\left<\mu_k,f_m\right>\\&=\left<m\diamond \mu_K,f\right>
	\end{align*}
\end{proof}
Now, we have all the required tools to prove that for an  abelian compact group $G$ and a reflexive Banach algebra $\mA$, the generalized group algebra $L^1(G,\mA)$ is left Strongly Arens irregular.
\begin{theorem}
Let $G$ be a compact  group and $\mA$ be a unital reflexive Banach algebra. Then, $$Z(L^1(G,\mA)^{**})=L^1(G,\mA).$$
\end{theorem}
\begin{proof}
	The inclusion $L^1(G,\mA)\subset Z(L^1(G,\mA)^{**})$ holds trivially true. To prove the reverse inclusion, let $m\in Z(L^1(G,\mA)^{**})$ and $\nu_m\in M_r(G,\mA)$ denote its restriction to $C(G,\mA^*)$. Due to \cref{imp1}, it will be sufficient to show that $\nu_m\in L^1(G,\mA)$. Let $B$ be a compact subset of $G$ such that $\mu(B)=0$. Then we choose a decreasing sequence of open sets $U_n\supset B$ such that $(\mu+|\nu_m|)(U_n\setminus B)\to 0$. By induction, we construct a sequence $\{\phi_n\}$ in $C(G)$ such that $0\leq \phi_n\leq 1$, $\phi_n(x)=1$ for $x\in B$ and $\phi_n(x)=0$ for $x\notin U_n\cap V_{n-1}$ (where $V_0=G$,$V_n=\{y:\phi_n(y)\neq 0\}$, $n=1,2,...$). For each $n$, $$d_n(x,y)=||x\phi_n-y\phi_n||_\infty$$ defines a continuous pseudo metric on $G$, and $K=\cap_{n=1}^\infty K_n$, then $G/K$ is metrizable and hence $m\diamond \mu_K\in L^1(G,\mA)$ by \cref{metrizable}. Consequently, by \cref{periodic} $\left<\nu_m,f\right>=\left<m\diamond \mu_K,f\right>$ for each right $K$-periodic function function $f\in L^\infty(G,\mA^*)$. Since, $\{V_n\}$ is decreasing, $\mu(V_n)\to \mu(B)=0$. Further, for each $a^*\in \mA^*$, the function $\chi_{V_N}\otimes a^*$ is right $K$-periodic. Thus,
	$$\nu_{m,a^*}(V)=\left<\nu,\chi_{V_n}\otimes a^*\right>=\left<m\diamond\mu_K,\chi_{V_n}\otimes a^*\right>\to 0\hspace{0.3in}(\because m\diamond\mu_K\in L^1(G,\mA))$$
	Since, $B\subset V_n\subset U_n$, and we have $\nu_{m,a^*}(V_n)\to \nu_{m,a^*}(B)$. Thus, $\nu_{m,a^*}(B)=0$ for each $a^*\in \mA^*$. Hence, $\nu_m(B)=0$ and due to regularity of $\nu_m$, we conclude that $\nu_m\ll \mu$, i.e $\nu_m\in L^1(G,\mA)$ and required.
\end{proof}
\begin{cor}
	Let $G$ be a compact abelian group and $\mA$ be any reflexive Banach algebra(not necessarily unital), then $L^1(G,\mA)$ is SAI.
\end{cor}
\begin{proof}
	The preceeding theorem combined with the \cref{unitization1} proves the assertion.
\end{proof}
\section{Bidual of $L^1(G,\mA)$ for non-reflexive Banach algebra $\mA$.}
As we have noticed in the previous sections that when $\mA^*$ does not have RNP, then the dual $L^1(G,\mA)^*$ strictly contains a copy of $L^\infty(G,\mA^*)$. This makes it difficult to access the topological center. It would be too demanding to expect $L^1(G,\mA)$ to be SAI in such cases, even when $\mA$ itself is SAI. But pseudo-center seems to be accessible in certain cenarios. We shall see that in certain situations, the elements of pseudo center can be identified with $Z(\mA^{**})$-valued measures.\\

	Let $\mA$ be a Banach algebra and $a^*\in \mA^*$ be a fixed element. For each $m\in Z_s(L^1(G,\mA)^{**})$, there is an associated map $\Delta_{m,a^*}: C(G)\to \mA^*$ defined as $\Delta_{m,a^*}(f)(a)=m(f\otimes {_a}a^*)$.

\begin{lemma} \label{awesome}
	Let $G$ be compact Hausdorff group, $\mA$ be a Banach algebra,and $m\in Z_s(L^1(G,\mA)^{**})$be such that $\Delta_{m,a^*}$ is compact for every $a^*\in \mA^*$, then the restriction of $m$ to $C(G,\mA^*)$ is a $Z(\mA^{**})$-valued measure.\end{lemma}
\begin{proof}
	Let $m\in Z_s(L^1(G,\mA)^{**})$ be any any arbitrary element. We denote the restriction of $m$ to $C(G,\mA^*)$ by the measure $\mu_m\in M_r(G,\mA^{**})$. Due to regularity of $\mu_m$, it would be sufficient to show that $\mu_m(E)\in Z(\mA^{**})$ for any open subset $E$ of $G$. Let $E$ by an open susbset of $G$.  For any $a^*\in\mA^*$ and $a^{**}\in \mA^{**}$, we have
	\begin{align*}
		\mu_m(E)\Hsquare a^{**}(a^*)&=\mu_m(E)({_{a^{**}}}a^*)
	\end{align*}
	Let $b^*={_{a^{**}}}a^*$. Invoking, the Riesz Representation theorem \cref{RRT1}, we have 
	\begin{align*}
		\mu_m(E)\Hsquare a^{**}(a^*)=\mu_{m,b^*}\left(E\right)
	\end{align*}
	where $\mu_{m,b^*}$ is the regular borel measure corresponding to the linear functional $L_{m,b^*}:C(G)\to \mathbb C$ defined as $L_{m,b^*}(f)=m(f\otimes b^*)$. Then there exists a compactly supported continuous function $f$ such that $0\leq f\preceq \chi_E$ and  $\left|\mu_{m,b^*}(E)-\int fd\mu_{m,b^*}\right|<\epsilon$. But $\int fd\mu_{m,b^*}=m(f\otimes b^*)=m(f\otimes {_{a^{**}}}a^*)$. Now notice that for any $\phi\in L^1(G,\mA)$, we have  \begin{align*}
		(f\otimes {_{a^{**}}}a^*)(\phi)&=\int\left<f(s){_{a^{**}}}a^*,\phi(s)\right>dt\\&=\left<{_{a^{**}}}a^*,\int f(s)\phi(s)dt\right>
	\end{align*}
	Using the module version of Cohen's factorization theorem \cref{Cohen}, we know that $L^1(G)\ast L^\infty(G)=C_{lu}(G)$ and since compactly supported continuous functions are left/right uniformly continuous, we can choose $g\in L^1(G)$ and $h\in L^\infty(G)$ such that $g\ast \check{h}= \check{f}$. Let $n$ denotes any Hahn-Banach extension of $g\otimes a^{**}$ to $L^1(G,\mA)^{**}$. Now for any any net $\{a_\gamma\}_{\gamma\in \wedge}$ in $\mA$ converging to $a^{**}$ in the $w^\ast$ topology of $\mA^{**}$, we see that $g\otimes a_\gamma$ converges to $g\otimes a^{**}$ in the $\sigma(L^1(G,\mA)^{**},L^\infty(G,\mA^*))$ topology and hence, \begin{align*}
		{_n}(h\otimes a^*)(\phi)&=n((h\otimes a^*)_\phi)\\&=\left<g\otimes a^{**},(h\otimes a^{*})_\phi\right>\\&=\lim_\gamma \left<g\otimes a_\gamma,(h\otimes a^*)_\phi\right>\\&=\lim_\gamma\int g(t)(h\otimes a^*)_\phi(t)(a_\gamma)dt\\&=\lim_\gamma\int g(t)\int\left<h(st)a^*,\phi(s)a_\gamma\right>dsdt\\&=\lim_\gamma\left<a^*,\left(\int g\ast \check{h}(s^{-1})\phi(s)ds\right) a_\gamma\right>\\&=\lim_\gamma\left<a^*,\left(\int f(s)\phi(s)ds\right) a_\gamma\right>\\&=\left<{_{a^{**}}}a^*,\int f(s)\phi(s)ds\right>
	\end{align*}
	Thus, we see that ${_n}(h\otimes a^*)=f\otimes {_{a^{**}}}a^*$. And hence \begin{align*}
		\left|\mu_m(E)\Hsquare a^{**}(a^*)-m\Hsquare n(h\otimes a^*)\right|<\epsilon
	\end{align*}
	Now we turn to the second Arens product and show that it is also in the arbitrarily small neighbourhood of $m\Hsquare n(h\otimes a^*)$. Notice that \begin{align*}\mu_m(E)\diamond a^{**}(a^*)&=\lim_\gamma \mu_m(E)\Hsquare a_\gamma(a^*)\\&=\lim_\gamma \mu_{m,b_\gamma^*}(E)\\&=\lim_\gamma \int \chi_Ed\mu_{m,b_\gamma^*}
	\end{align*}
	where $b_\gamma^*={_{a_\gamma}}a^*$. Since, $\Delta_{m,a^*}$ is compact, its adjoint $\Delta_{m,a^*}^{*}:A^{**}\to M_r(G)$ given by $\Delta_{m,a^{*}}^*(a^{**})=\mu_{m,{_{a^{**}}a^*}}$ satisfies the property that for any bounded net $\{a_\lambda^{**}\}$ in $\mA^{**}$, the net $\{\Delta_{m,a^*}^*(a_\lambda^{**})\}$ converges in norm. Hence, $\lim_{\gamma}\int\chi_Ed\mu_{m,b_\lambda^*}=\mu_{m,b_\gamma^*}(E)$
	\begin{align*}
		\left|\mu_m(E)\diamond a^{**}(a^*)-\int fd\mu_{m,b^*}\right|\leq |\mu_{m,b^*}|(U\setminus E)<\epsilon\\&
	\end{align*}
	Thus, we see that $|\mu_m(E)\diamond a^{**}(a^*)-\mu_m(E)\Hsquare a^{**}(a^*)|<\epsilon$. Since $\epsilon$ was arbitrarily chosen, we conclude that $\mu_m(E)\diamond a^{**}(a^*)=\mu_m(E)\Hsquare a^{**}(a^*)$ for all $a^{**}\in \mA^{**}$ and all $a^*\in \mA^*$. Thus, $\mu_m(E)\in Z(\mA^{**})$. 
\end{proof}
\begin{rem}
	It is expected that if $\mA$ is SAI and holds RNP and WSC, then for any compact group $G$ and $m\in Z_s(L^1(G,\mA)^{**})$, the map $\Delta_{m,a^{*}}$ is compact for each $a^*\in \mA^*$. And it is possible to do computations similar to those in \cref{4} to conclude that $L^1(G,\mA)$ is pseudo-strongly irregular. For doing this, one might need  something  stronger than WSC of $L^1(G,\mA)$.\end{rem}
Although it is known that $L^1(S,X)$ is WSC for any WSC Banach space $X$ and finite measure space $S$, the addition constraint of RNP on $X$ results into something remarkably stronger. 
\begin{lemma}\label{WSC}
	If $X$ is WSC Banach space such that $X$ has RNP w.r.t a finite measure space $(S,\mathscr{A},\nu)$, then $L^1(S,X)$ is $\sigma(L^1(S,X),S^\infty(S,X^*))$-sequentially complete. In other words, if $\{\phi_n\}\}_{n\in \mathbb N}$ is sequence in $L^1(S,X)$ such that for each $f\in S^\infty(S,X^*)$ the sequence $\{\phi_n(f)\}_{n=1}^\infty$ is Cauchy, then there exists $\phi\in L^1(S,X)$ such that $\phi_n\to\phi$ in the $\sigma(L^1(S,X),S^\infty(S,X^*))$-topology.
\end{lemma}
\begin{proof}
	Suppose $\{\phi_n\}_{n\in \mathbb N}$ is a sequence in $L^1(S,X)$ such that for each $f\in S^\infty(S,X^*)$ the sequence $\{\phi_n(f)\}_{n=1}^\infty$ is Cauchy. Using the uniform boundedness principle and the fact that $S^\infty(S,X^*)$ is norming for $L^1(S,X)$, we can conclude that the sequence $\{\phi_n\}_{n=1}^\infty$ is norm bounded. For any subset $E$ of $S$ with $\nu(E)\neq 0$ and any $x^*\in X^*$, consider the function $\chi_E\otimes x^*\in S^\infty(S,X^*)$. Then $\{\left<\phi_n,\chi_E\otimes x^*\right>\}$ is a cauchy sequence and \begin{eqnarray*}
		\left<\phi_n,\chi_E\otimes x^*\right>&=&\int_E\left<\phi_n(t),x^*\right>dt\\
		&=&x^*\left(\int_E\phi_n(t)dt\right)
	\end{eqnarray*}
	Since, above equality holds for each $x^*\in X^*$, we conclude that $\{\int_E\phi_n(t)dt\}$ is a weakly cauchy sequence in $X$. Since $X$ is WSC, there must exists $a_E\in X$ such that $\int_E\phi_n(t)dt\to a_E$ weakly. Now we define a $X$ valued measure $F:\mathscr A\to X$ as $F(E)=a_E$. Clearly $F$ is a measure of bounded variation. Further, for each $x^*\in X^*$ we see that $x^*F$ is a finite measure wich is a limit of a sequence of measures absolutely continuous w.r.t to $\mu$ and hence, $x^*F$ is countably additive due to Vitali-Hahn-Saks theorem for each $x^*\in X^*$. Thus, $F$ is weakly countably additive and hence countably additive due to Orlicz-Pettis theorem(\cite[Sec. I.4,Corr.4]{Diestel}). Further, $F\ll\mu$ can be seen from the fact that $F$ vanishes on $\mu$-null sets. Thus, by RNP there exists $\phi\in L^1(S,X)$ such that $F(E)=\int_E \phi(t)d\mu(t)$. Thus, we have that $$\lim_n\left<\phi_n,\chi_E\otimes x^*\right>=\left<\phi,\chi_E\otimes x^*\right>$$ for each borel subset $E$ of $G$ and each $x^*\in X^*$. Since, functions of the type $\chi_E\otimes x^*$ spans $S^\infty(S,X^*)$, we conclude that $\phi_n\to \phi$ in $\sigma(L^1(S,X),S^\infty(S,X^*))$-topology.
\end{proof}

\end{document}